\documentclass[12pt,a4paper,reqno]{amsart}
\usepackage{graphicx} 
\usepackage{tikz}
\usepackage{mathtools}
\mathtoolsset{showonlyrefs}
\usepackage[colorlinks=true, linkcolor=blue]{hyperref}
\usepackage{orcidlink}
\title[On discrete holomorphic Paley--Wiener spaces and sampling]{On discrete holomorphic Paley--Wiener spaces \\and sampling on the square lattice}

\date{}
\usepackage{amsfonts,amsmath,amssymb, amsthm}
\usepackage[a4paper,left=2cm,right=2cm]{geometry}

\DeclareMathOperator*{\supp}{supp}

\DeclareMathOperator{\sinc}{sinc}

\newcommand{\bZ}{\mathbb Z}

\newcommand{\leqs}{\leqslant}

\newtheorem{thm}{Theorem}[section]
\newtheorem{lem}[thm]{Lemma}
\newtheorem{prop}[thm]{Proposition}
\newtheorem{cor}[thm]{Corollary}

\theoremstyle{definition}
\newtheorem{rmk}[thm]{Remark}

\theoremstyle{definition}\newtheorem{defn}[thm]{Definition}
\newtheorem*{thm*}{Theorem}

\usepackage{tcolorbox}

\DeclareMathOperator{\id}{Id}

\address{\phantom{i}$^1$Dipartimento di Ingegneria Gestionale, dell'Informazione e
della Produzione, Universit{\`a} degli Studi di Bergamo, Viale G. Marconi 5, 24044,
Dalmine BG, Italy}
\email{alessandro.monguzzi@unibg.it}
\author[A. Monguzzi]{A. Monguzzi \orcidlink{0000-0003-3233-5000}}
\author[M. Monti]{M. Monti\orcidlink{0000-0001-6848-5938}}
\address{Dipartimento di Ingegneria Gestionale, dell'Informazione e
della Produzione, Universit{\`a} degli Studi di Bergamo, Viale G. Marconi 5, 24044,
Dalmine BG, Italy}
\email{alessandro.monguzzi@unibg.it}
\email{matteo.monti@unibg.it}
\thanks{MSC 2020: 30G25; 39A12; 30D15; 30H99}
\thanks{All the authors are members of Indam--Gnampa and are supported by the PRIN 2022 project "TIGRECO - TIme-varying signals on Graphs: REal and COmplex methods" funded by the MUR (Ministero dell'Università e della Ricerca), Grant\textunderscore 20227TRY8H, CUP\textunderscore F53D23002630001.}
\keywords{Discrete complex analysis, Paley--Wiener spaces, sampling}

\usepackage{mathtools}
\newcounter{mysubequations}

\begin{document}
\begin{abstract}
We consider a reproducing kernel Hilbert space of discrete entire functions on the square lattice $\mathbb Z^2$ inspired by the classical Paley--Wiener space of entire functions of exponential growth in the complex plane. For such space we provide a Paley--Wiener type characterization and a sampling result. 
\end{abstract}

\maketitle

\section{Introduction}
Discrete holomorphicity was distinctively introduced about 80 years ago  in \cite{Isa, Isa2}  and again in \cite{Fer} and later studied extensively by  Duffin and collaborators \cite{D, D2, DP} in the setting of the square lattice and of rhombic lattices. After its introduction there was
 a steady small production of papers on discrete holomorphicity and related areas, but a major interest in the topic remained dormant for several years. It is in the last 20 years or so that the study of discrete holomorphic functions has regained the attention of several mathematicians.
 Kenyon in \cite{Kenyon} and Smirnov, Chelkak and collaborators in a series of papers  \cite{CS, CS2, C, CLR} resumed the studies of Duffin on rhombic lattices (equivalently, isoradial graphs) and carefully developed techniques of discrete complex analysis with a view to important applications in probability and statistical physics. Mercat, on the other hand,  extended the theory to the setting of discrete Riemann surfaces \cite{Mercat, Mercat2}. We also recall, among others, the interesting papers \cite{Skopenkov, BG}, whereas for a general introduction to the topic of discrete complex analysis we refer the reader to the expository papers \cite{Smirnov} and \cite{L} and the references therein.
 
 One area that seems to be largely unexplored in discrete complex analysis is the one of function spaces of discrete holomorphic functions. A major obstacle certainly is the fact that the pointwise product of two discrete holomorphic functions is no longer holomorphic. So there is no standard way to proceed and new ideas and tools are needed. See for example \cite{Alpay}, where two different notions of product are introduced and studied.  There are some interesting papers from the 1970s by Zeilberger and collaborators on (partial) discrete versions of Paley--Wiener type theorems on entire functions of exponential growth \cite{ZD, Z2, Z3, Z4} and some other papers scattered in the literature (for instance \cite{Bolen, DM}), but there has not yet been a systematic study of discrete holomorphic function spaces.

Building on the results of Zeilberger and collaborators we study a reproducing kernel Hilbert space of discrete entire functions on the lattice $\mathbb Z^2$ which is reminiscent of the classical Paley--Wiener space of entire functions of exponential growth. For such a space we give a characterization in terms of the Fourier transform of its elements and prove a sampling result. We would like to point out here that one of the goal we had in mind when we started this project was to improve Zeilberger's results and to define and study a Paley--Wiener space on $\mathbb Z^2$ relying as much as possible on discrete complex analysis, thus differing from Pesenson's approach in \cite{Pesenson}, where he defines Paley--Wiener spaces on combinatorial graphs via the spectral theory of the classical Laplace operator. It was therefore a nice surprise to discover in retrospect that the space we consider on $\mathbb Z^2$ and the one considered by Pesenson are, although different, closely related. We will come back to this in Section \ref{s:sampling}.

A function $F:\mathbb Z^2\to \mathbb C$ is said to be \emph{discrete entire} if the identity 
\begin{equation}\label{eq:holomorphicity}
F(m+1,n+1)-F(m,n)=-i\big(F(m,n+1)-F(m+1,n)\big)
\end{equation}
holds true for every $(m,n)\in\mathbb Z^2$. The above identity can be easily reformulated in terms of the real and imaginary parts of the function $F$, yielding a pair of equations reminiscent of the classical Cauchy--Riemann equations. Since this is not necessary for our purpose we avoid to do it and instead refer the reader to \cite{Smirnov}.  If we embed $\mathbb Z^2$ in the complex plane, it is immediate to verify that the restrictions of the entire functions $z$ and $z^2$ to the Gaussian integers are discrete entire. However, this is no longer true for the function $z^3$. Hence, the pointwise product of two discrete entire functions is not discrete entire in general. This is certainly a major drawback of discrete holomorphicity. 

Before stating our results, we recall the definition of discrete entire exponential functions given in \cite{ZD}. In the following we identify the one-dimensional torus $\mathbb T=\mathbb R/2\pi\mathbb Z$ with the interval $[-\pi,\pi)$.
\begin{defn}\label{d:exponential}
For every $t\in\mathbb T, t\neq\pm \pi/2$ we call \emph{discrete exponential} the function $e_t:\mathbb Z^2\to \mathbb C$, 
\begin{equation}\label{d:discrete-exponential}
e_{t}(m,n)=e^{itm}\bigg(\frac{1+ie^{it}}{i+e^{it}}\bigg)^n.
\end{equation}
\end{defn}
Here we limit ourselves to noting that the restriction of $e_t(m,n)$ at height $n=0$ is the actual exponential function and to pointing out that $e_t$ is discrete entire on $\mathbb Z^2$. For the origin of such a function we refer the reader to Section \ref{s:preliminaries}.

Now that we have discrete exponentials, we can consider discrete entire functions of \emph{discrete exponential growth}. This was first done by Zeilberger and Dym in \cite{ZD}, where they proved the following result.
\begin{thm}[Zeilberger--Dym, {\cite{ZD}}]\label{thm:ZD}
Let $0<\alpha<\pi/2$. Let $F$ be an entire function  with the property that for every $k\in\mathbb N$ there exists a constant $c_{k}>0$ such that
\begin{equation}\label{F:decay-stronger}
|F(m,n)|\leqslant c_{k} (1+|m|)^{-k}|e_{-\alpha}(0,1)|^{|n|}.
\end{equation}
Then, there exists a function $f$ in $C^\infty(\mathbb T)$ supported in
\medskip
\[
D_\alpha=\big\{t: |t|\leqslant\alpha\big\}\cup \big\{t: \pi-\alpha\leqslant|t|\leqslant\pi\big\}
\]
such that the identity 
\begin{equation*}
F(m,n)=\frac{1}{2\pi}\int_{\mathbb T} f(t) e_t(m,n)\, dt
\end{equation*}
holds true.
\end{thm}
 We  also recall the paper \cite{Z4}, where a  notion of exponential growth different than the one above  is used.  A few comments are in order. First, notice that equation \eqref{F:decay-stronger} guarantees a ``discrete exponential bound'' on the growth of the function $F$ and that such condition can be rewritten as
\[
|F(m,n)|\leqslant c_{k} (1+|m|)^{-k}\bigg(\frac{\cos\alpha}{1-\sin\alpha}\bigg)^{|n|}.
\]
Second, notice that \eqref{F:decay-stronger} also demands a fast horizontal decay of the function $F$ at every height $n\in\mathbb Z$. Such assumption guarantees in particular that the function
\[
f_n(t)=\sum_{m\in\mathbb Z} F(m,n)e^{-imt}
\]
is a well-defined smooth function on the torus $\mathbb T$ for every $n\in\mathbb Z$. 
We refer the reader to Theorem \ref{thm:PW_decay_ascoli} and Theorem \ref{thm:decay_zero} for results under weaker decay properties at every height. In what follows we state a partial converse of Theorem \ref{thm:ZD}, where  a growth condition weaker than \eqref{F:decay-stronger} appears.

\begin{thm}\label{thm:PW}
Let $0<\alpha<\pi/2$. Let $f$ be a function in $C^\infty(\mathbb T)$ with
$\supp f\subseteq D_\alpha$
and set 
\begin{equation}\label{eq:PW2}
F(m,n)=\frac{1}{2\pi}\int_{\mathbb T} f(t) e_t(m,n)\, dt.
\end{equation}
Then,  $F$ is a discrete entire function with the property that for every $k\in\mathbb N$ and every $\varepsilon>0$ with $\alpha+\varepsilon<\pi/2$ there exists a constant $c_{k,\varepsilon,\alpha}>0$ such that
\begin{equation}\label{F:decay}
|F(m,n)|\leqslant c_{k,\varepsilon,\alpha} (1+|m|)^{-k}|e_{-(\alpha+\varepsilon)}(0,1)|^{|n|}.
\end{equation}
 \end{thm}

\begin{rmk}
 The exponential growth conditions \eqref{F:decay-stronger} and \eqref{F:decay}
  play the role of the classical continuous growth conditions of functions of exponential type $\alpha$. Namely, an entire function $f$ is of exponential type $\alpha$ if there exists a constant $c>0$ such that 
  \[
  |f(z)|\leqs  c e^{\alpha|z|}.
  \]
  Alternatively, one may find in the literature the condition that, for all $\varepsilon>0$, there exists $c_\varepsilon>0$ such that
  \[
  |f(z)|\leqs c_\varepsilon e^{\alpha+\varepsilon}|z|.
  \]
  These two conditions turn out to be equivalent for $L^p$ entire functions (see, e.g., \cite[Remark 2]{Andersen}). Thus, a natural question, beyond the goal of the present paper, would be to investigate if the same holds true for the discrete conditions \eqref{F:decay-stronger} and \eqref{F:decay}. 
 \end{rmk}

We now focus on an $L^2$ version of the above theorems. From now on $\alpha$ will always be such that $0<\alpha<\pi/2$. Consider an entire function $F$ on $\mathbb Z^2$ satisfying the following properties,
\begin{align}\label{eq:PW_p_i}
&(i)\quad \exists c>0 \textrm{ such that } |F(m,n)|\leqslant c |e_{-\alpha}(0,1)|^{|n|};\\
\label{eq:PW_p_iii}
&(ii)\quad \sum_{m\in\mathbb Z}|F(m,0)|^2<+\infty
\end{align}
and define the Paley--Wiener space $PW_{\alpha}$ as
\begin{equation*}
PW_\alpha=\bigg\{F:\mathbb Z^2\to\mathbb C: F \textrm{ is entire and satisfies }\eqref{eq:PW_p_i}\textrm{ and } \eqref{eq:PW_p_iii} 
\bigg\}
\end{equation*}
endowed with the norm
\begin{equation}\label{PW_norm}
\|F\|^2_{PW_\alpha}= \sum_{m\in\mathbb Z} |F(m,0)|^2.
\end{equation}
We will prove that $PW_\alpha$ is a reproducing kernel Hilbert space, we will compute its reproducing kernel and we will prove a sampling result. At this stage is not even clear that \eqref{PW_norm} defines an actual norm on $PW_\alpha$.
We prove the following characterization of $PW_\alpha$ (cfr. \cite{Z4, Mugler}).
\begin{thm}\label{thm:PW-vero}
Let $F$ be a function in $PW_\alpha$. Then there exists a function $f$ in $L^2(\mathbb T)$ with $\supp f\subseteq D_\alpha$
such that the identity
\begin{equation}\label{eq:PW}
F(m,n)=\frac{1}{2\pi}\int_{\mathbb T} f(t) e_t(m,n)\, dt
\end{equation}
holds true. Moreover,
\begin{equation}\label{eq:isometry}
\|F\|_{PW_\alpha}=(2\pi)^{-1/2}\|f\|_{L^2(\mathbb T)}.
\end{equation}
Conversely, let $f\in L^2(\mathbb T)$ be such that $\supp f\subseteq D_\alpha$. Then, the function $F$ defined by \eqref{eq:PW} is in $PW_\alpha$ and \eqref{eq:isometry} holds true. 
\end{thm}

We point out here that Paley--Wiener type spaces in discrete setting, namely on combinatorial graphs, have been investigated via spectral theory by Pesenson in \cite{Pesenson}. We will compare the space $PW_\alpha$ with Pesenson's results in Section \ref{s:remarks}.


Notice that, unlike in Zeilberger and Dym result, we do not explicitly require any decay at every height $n\in\mathbb Z$. 
However, the growth condition \eqref{eq:PW_p_i} together with the summability condition \eqref{eq:PW_p_iii} actually imply that
\[
\lim_{|m|\to+\infty}F(m,n)=0,\quad  n\in\mathbb Z,
\]
see Theorem \ref{thm:PW_decay_ascoli}. This last observation is crucial for the proof of Theorem \ref{thm:PW-vero}.

If we require a function $F$ of $PW_\alpha$ to be of rapid decay at height $n=0$ instead of being only in $L^2$, we obtain
the following stronger result for the decay of $F$ at every height. 
\begin{thm}\label{thm:decay_zero}
Let $F$ be in $PW_\alpha$.  
Assume that for every $k\in\mathbb N$ there exists a constant $c_{k}$ such that
\begin{equation}\label{eq:decay_zero}
|F(m,0)|\leqslant c_{k}(1+|m|)^{-k}.
\end{equation}
Then,
\[
\sum_{m\in\mathbb Z}|F(m,n)|<+\infty
\]
for every $n\in\mathbb Z$.
\end{thm}

We conclude by proving a sufficient condition for a sequence $\Lambda\subseteq \mathbb Z$ to be a sampling sequence for $PW_\alpha$. 
Let $\Lambda$ be sequence of integers and set  $\Lambda\cap 2\mathbb{Z}=\{p_k\}_k$ and $\Lambda\cap (2\mathbb{Z}+1)=\{q_k\}_k$. Assume also that $\{p_k\}_k$ and $\{q_k\}_k$ are ordered and set
\begin{equation}\label{eq: delte}
    \delta_e=\sup_k(p_{k+1}-p_k);\qquad \delta_o=\sup_k(q_{k+1}-q_k).
\end{equation}
\begin{thm}\label{thm: Sampling}
If $\max(\delta_e,\delta_o)<\pi/\alpha$, 
then $\Lambda$ is a sampling  sequence for $PW_\alpha$. Namely, there exists $c>0$ such that, for every function $F$ in $PW_\alpha$, 
\begin{equation}\label{eq:sampling}
c\|F\|^2_{PW_\alpha}\leqslant \sum_{\lambda\in\Lambda}|F(\lambda,0)|^2\leqslant \|F\|^2_{PW_\alpha}. 
\end{equation}
\end{thm}
The above condition  involving $\delta_e$ and $\delta_o$ may be a little unexpected. We will come back to this in Section \ref{s:sampling}.

The paper is organized as follows. In Section \ref{s:preliminaries} we recall some preliminaries, in Section \ref{s:PW} we deal with the characterization of the space $PW_\alpha$ and Theorem \ref{thm:decay_zero}, whereas Section \ref{s:sampling} is devoted to the sampling result. We conclude with some final remarks in Section \ref{s:remarks}.

\section{Discrete exponentials and contour integrals}\label{s:preliminaries}
Here we briefly see how discrete exponentials arise and recall a useful result about a contour-type integral of discrete holomorphic functions.  In the following we sometimes identify $\mathbb Z^2$ and $\mathbb Z+i\mathbb Z$ even if not explicitly stated. The following construction of discrete exponentials is taken from \cite{ZD} and we include it here for the convenience of the reader.

We want a discrete entire function $e_t(m,n)$ which extends the classical exponential function defined on the integer $\mathbb Z$ to the discrete plane of the Gaussian integers. Namely, we are looking for a function $e_t(m,n)$  such that, for every $t\in [-\pi,\pi)$,
\begin{equation}\label{phi}
e_t(m,0)=e^{imt}.
\end{equation}
Moreover, in 
analogy with the continuous case $e^{z}=e^{x}e^{iy}$, we would like  such an extension to be of the form
\[
e_t(m,n)= e^{imt}\varphi_t(n)
\]
for some suitable function $\varphi_t$.
By \eqref{phi}, we necessarily require that $\varphi_t(0)=1$. Then, imposing the holomorphicity condition to $e_t(m,n)$, one gets
 \begin{align*}
e^{i(m+1)t} \varphi_t(n+1)-e^{imt}\varphi_t(n)=-i\big(e^{imt}\varphi_t(n+1)-e^{i(m+1)t}\varphi_t(n)\big).
\end{align*}
Equivalently,
\[
\varphi_t(n+1)(e^{it}+i)=\varphi_t(n)(1+ie^{it}).
\]
Notice that the above identity does not actually depend on $m$. By recursion and the assumption on $\varphi_t(0)$, one finally obtains
\begin{equation}\label{phi-explicit}
\varphi_t(n)=\bigg(\frac{1+ie^{it}}{i+e^{it}}\bigg)^n,\quad n\in\mathbb Z,\, t\neq\pm\frac{\pi}{2}.
\end{equation}

We conclude the section by recalling a useful contour-type integral for discrete holomorphic functions. Let $\Gamma\subseteq \mathbb Z+i\mathbb Z$ be a discrete contour with vertices $z_0,z_1,\ldots, z_m$ where $|z_{n}-z_{n+1}|=1$. In \cite{D} the following integral is defined,
\begin{align}\label{eq:contour_integral}
\int_{\Gamma} F:G\, =\frac{1}{4}\sum_{n=0}^{m} \big(F(z_n)+F(z_{n+1})\big)\big(G(z_n)+G(z_{n+1})\big) (z_n-z_{n+1}).
\end{align}
The following result holds true.
\begin{prop}[{\cite[Section 3]{D}}]\label{p:integral}
Let $F,G$ be two discrete entire functions and let $\Gamma$ be a closed discrete contour. Then,
\[
\int_{\Gamma} F:G=0.
\]
\end{prop}

\section{Paley--Wiener spaces}\label{s:PW}
In this section we prove Theorem \ref{thm:PW} and that $PW_\alpha$ is a normed vector space. We then prove Theorem \ref{thm:PW-vero} and that $PW_\alpha$ is a reproducing kernel Hilbert space. We conclude the section by proving Theorem~\ref{thm:decay_zero}. 

\begin{proof}[Proof of Theorem \ref{thm:PW}] 
The fact that $F$ is an entire function is obvious from \eqref{eq:PW2} and the holomorphicity of the discrete exponential. Recall that
\[
\varphi_t(n)= \bigg(\frac{1+ie^{it}}{i+e^{it}}\bigg)^n=\bigg(\frac{\cos t}{1+\sin t}\bigg)^n 
\]
and set $f_n(t)=f(t)\varphi_t(n)$. Then, $f_n\in C^\infty(\mathbb T)$ with $\supp f_n\subseteq D_\alpha$.
Integrating by parts, for every $k\in\mathbb N$, we have
\begin{align}\label{estimate-derivatives}
\begin{split}
|F(m,n)|&= |m|^{-k}\bigg|\frac{1}{2\pi}\int_{\mathbb T} f^{(k)}_n(t) e^{imt}\, dt\bigg|\leqs c_k(1+|m|)^{-k}\| f^{(k)}_n\|_{L^\infty(\mathbb T)}
\end{split}
\end{align}
since
\begin{equation}\label{estimate-derivatives-f}
f^{(k)}_n(t)=\sum_{j=0}^{k}\binom{k}{j}f^{(k-j)}(t) \frac{d^j}{dt^j}(\varphi_t(n))
\end{equation}
and $\supp f^{(k-j)}\subseteq D_\alpha$ for every $j=0,\ldots, k$. Let us consider the extension to the complex plane of $\varphi_t(n)$, that is,
\[
\varphi_z(n)=\bigg(\frac{\cos(z)}{1+\sin(z)}\bigg)^n.
\]
Since $0<\alpha<\pi/2$, such function is holomorphic in the strip $\{z=x+iy, |x|\leqs \alpha+\varepsilon\}$ where $0<\varepsilon<\pi/2-\alpha$. Let $\gamma_{\alpha,\varepsilon}$ be the circumference parameterized by $((\alpha+\varepsilon)\cos\theta,(\alpha+\varepsilon)\sin\theta)$, $\theta \in [0,2\pi)$. Then, by Cauchy's formula, for every $t\in [-\alpha,\alpha]$ we have
\begin{align*}
\frac{d^j}{dt^j}(\varphi_t(n))= \frac{j!}{2\pi}\int_{\gamma_{\alpha,\varepsilon}}\frac{\varphi_\zeta(n)}{(\zeta-t)^{j+1}}\, d\zeta.
\end{align*}
Hence,
\begin{align*}
\bigg|\frac{d^j}{dt^j}(\varphi_t(n))\bigg|\leqs j!(\alpha+\varepsilon)\varepsilon^{-j-1}\sup_{\zeta\in\gamma_{\alpha,\varepsilon}}|\varphi_\zeta(n)|.
\end{align*}
Thus, we want to maximize the constrained function
\[
|\varphi_\zeta(n)|=\bigg|\frac{\cos(x+iy)}{1+\sin(x+iy)}\bigg|^n=\bigg(\frac{\cosh y-\sin x}{\cosh y+\sin x}\bigg)^{\frac n2}, \qquad |\zeta|^2=x^2+y^2=(\alpha+\varepsilon)^2.
\]
One may verify that for $n>0$ the constrained maximum is obtained for $(x,y)=(-(\alpha+\varepsilon),0)$, whereas for $n<0$ it is obtained for $(x,y)=(\alpha+\varepsilon,0)$. In conclusion, for every $n\in\mathbb Z$, we have
\[
\sup_{|\zeta|^2=(\alpha+\varepsilon)^2}|\varphi_\zeta(n)|= \bigg(\frac{\cos(\alpha+\varepsilon)}{1-\sin(\alpha+\varepsilon)}\bigg)^{|n|}=|e_{-(\alpha+\varepsilon)}(0,1)|^{|n|}
\]
Hence, for every $\varepsilon$ with $\alpha+\varepsilon<\pi/2$, we have
\[
\sup_{t\in[-\alpha,\alpha]}\bigg|\frac{d^j}{dt^j}(\varphi_t(n))\bigg|\leqs j! \varepsilon^{-j-1}(\alpha+\varepsilon)\bigg(\frac{\cos(\alpha+\varepsilon)}{1-\sin(\alpha+\varepsilon)}\bigg)^{|n|}.
\]
Therefore, from  \eqref{estimate-derivatives} and \eqref{estimate-derivatives-f}, we get
\begin{align*}
\sup_{t\in[-\alpha,\alpha]}|f^{(k)}_n(t)|&\leqslant \max_{j=0,\ldots, k} \|f^{(j)}\|_{L^\infty} \sum_{j=0}^{k}\binom{k}{j}\sup_{t\in[-\alpha,\alpha]}\bigg|\frac{d^j}{dt^j}(\varphi_t(n))\bigg|\\
&\leqslant \max_{j=0,\ldots, k} \|f^{(j)}\|_{L^\infty}\sum_{j=0}^{k}\binom{k}{j} j! \varepsilon^{-j-1}(\alpha+\varepsilon)\bigg(\frac{\cos(\alpha+\varepsilon)}{1-\sin(\alpha+\varepsilon)}\bigg)^{|n|}\\
&\leqslant c_{k,\varepsilon, \alpha}|e_{-(\alpha+\varepsilon)}(0,1)|^{|n|}.
\end{align*}
We still have to prove the same estimate for $t$ such that $\pi-\alpha\leqs |t|\leqs \pi$. By periodicity, we can actually think to have $t>0$ such that $\pi-\alpha<t<\pi+\alpha$. Repeating the same argument above but this time using the circumference $(\pi+(\alpha+\varepsilon)\cos\theta, \pi+(\alpha+\varepsilon)\sin\theta)$ we obtain the same estimate. In conclusion, from \eqref{estimate-derivatives} we get
\begin{align*}
|F(m,n)|&\leqs c_{k,\varepsilon,\alpha}(1+|m|)^{-k}\bigg(\frac{\cos(\alpha+\varepsilon)}{1-\sin(\alpha+\varepsilon)}\bigg)^{|n|},
\end{align*}
as we wished to show. 
\end{proof}
\subsection{Proof of Theorem \ref{thm:PW-vero}}
We lay down some preliminary results to eventually prove Theorem \ref{thm:PW-vero}. The following simple lemma will be of great use.
\begin{lem}\label{lem:formula} 
Let $F$ be an entire function. Then, for every  $m\in\mathbb Z^+$ and $n\in\mathbb Z$, we have
\begin{equation}
\label{formula_strati}
F(m,n)=(-i)^m \Big(F(0,n)+2i\sum_{k=1}^{m-1}i^k F(k,n-1)+iF(0,n-1)+i^{m+1}F(m,n-1)\Big).
\end{equation}
\end{lem}
\begin{proof}
The desired identity is a direct consequence of holomorphicity and an induction argument. For $m=1$ the identity \eqref{formula_strati} reduces to
\[
F(1,n)=-i\Big(F(0,n)+iF(0,n-1)-F(1,n-1)\Big)
\]
and this is the very definition of holomorphicity in a square of the grid. Hence, it holds true because $F$ is entire. Assume now that \eqref{formula_strati} holds true for $m-1$ and let us prove it for $m$. We have
\begin{align*}
F(m,n)=&F(m-1,n-1)+iF(m,n-1)-iF(m-1,n)\\ 
=&F(m-1,n-1)+iF(m,n-1)\\
&+(-i)^m \Big(F(0,n)+2i\sum_{k=1}^{m-2}i^k F(k,n-1)+iF(0,n-1)+i^{m}F(m-1,n-1)\Big)\\
=& (-i)^m \Big(F(0,n)+2i\sum_{k=1}^{m-1}i^k F(k,n-1)+iF(0,n-1)+i^{m+1}F(m,n-1)\Big),
\end{align*}
as we wished to show.
\end{proof}

\begin{rmk}
The above lemma allows to express $F(m,n)$, for $m\in\mathbb Z^+$ and $n\in\mathbb Z$, using only the value of $F$ in $(0,n)$ and the values of $F$ at height $n-1$. We can easily obtain similar formulas that allow us to express $F(m,n)$  for each $m\in\mathbb Z$ using either the values of $F$ at height $n-1$ or $n+1$. Namely, consider the function 
\[
G(m,n)=(-1)^{m+n}F(-m,n).
\]
Such function is still discrete entire and, applying \eqref{formula_strati} to it, for every $m\in\mathbb Z^+$ and $n\in\mathbb Z$, we obtain
\begin{equation}\label{formula_strati_negativi}
F(-m,n)= i^m\Big(F(0,n)-2i\sum_{k=1}^{m-1} (-i)^k F(-k,n-1)-i F(0,n-1)+(-i)^{m+1} F(-m,n-1)\Big). 
\end{equation}
Similarly, if we apply \eqref{formula_strati} to the entire functions $H(m,n)=(-1)^{m+n}F(m,-n)$ and $L(m,n)=F(-m,-n)$, we get, for every $m\in\mathbb Z^+$ and $n\in\mathbb Z$, 
\begin{equation}\label{formula_strati_sup}
F(m,n)= i^m\Big(F(0,n)-2i\sum_{k=1}^{m-1}(-i)^k F(k,n+1)-iF(0,n+1)+(-i)^{m+1}F(m,n+1)\Big)
\end{equation}
and
\begin{equation}\label{formula_strati_sup_neg}
F(-m,n)=(-i)^m\Big(F(0,n)+2i\sum_{k=1}^{m-1}i^k F(-k,n+1)+iF(0,n+1)+i^{m+1}F(-m,n+1)\Big).
\end{equation}
\end{rmk}

We use Lemma \ref{lem:formula} to deduce that, for every $n\in\mathbb Z$, the set $\mathbb Z\times\{n\}$
is a set of uniqueness for functions in $PW_\alpha$. 

\begin{prop}\label{p:uniqueness}
Let $F$ be in $PW_\alpha$ and let $n\in\mathbb Z$ be fixed. Suppose that 
\[
F(m,n)=0 \quad \text{for every }\quad m\in\mathbb Z,
\]
Then, $F$ is identically zero.
\end{prop}

\begin{proof}

Set $A_n=\mathbb Z\times\{n\}$. We prove that if $F|_{A_n}$ is identically zero, then $F|_{A_{n+1}}$ and $F|_{A_{n-1}}$ are identically zero as well. Hence, it follows that $F=0$. 

Since $F|_{A_n}\equiv 0$, it follows from \eqref{formula_strati}  that, for every $m\in\mathbb Z^+$,
\[
F(m,n+1)=(-i)^m F(0,n+1). 
\]
Applying again \eqref{formula_strati} at height $n+2$ we get, for every $m\in\mathbb Z^+$,
\begin{align*}
&F(m,n+2)= (-i)^m\Big(F(0,n+2)+2i\sum_{k=1}^{m-1}i^k F(k,n+1)+iF(0,n+1)+i^{m+1}F(m,n+1)\Big)\\
&=(-i)^m \Big(F(0,n+2)+2i\sum_{k=1}^{m-1}i^k(-i)^kF(0,n+1)+iF(0,n+1)+i^{m+1}(-i)^mF(0,n+1)\Big)\\
&= (-i)^m\Big(F(0,n+2)+2imF(0,n+1) \Big).
\end{align*}
Therefore,
\begin{align*}
|F(m,n+2)|\geqslant 2m|F(0,n+1)|-|F(0,n+2)|.
\end{align*}
However, since $F\in PW_\alpha$, the function $F(\cdot, n+2)$ is bounded thanks to \eqref{eq:PW_p_i}. Therefore, passing to the limit $m\to\infty$, it turns out that $F(0,n+1)$ is necessarily $0$ and, by holomorphicity, we deduce that $F(m,n+1)=0$ for every $m\in\mathbb Z$ as we wished to show. The proof that $F$ vanishes identically on $A_{n-1}$ as well follows similarly using \eqref{formula_strati_sup} instead of  \eqref{formula_strati}.
\end{proof}

Notice that in the above lemma we did not really rely the exponential bound on the size of $F$, but rather only on the boundedness of the function that $F(\cdot,n)$ at every height $n\in\mathbb Z$. An immediate consequence of the above result is the following one.
\begin{cor}
$PW_\alpha$ is a normed vector space.
\end{cor}
\begin{proof}
The only thing that requires some attention is the fact that $\eqref{PW_norm}$ defines a norm on $PW_\alpha$. In particular, we need to prove that $\|F\|_{PW_\alpha}=0$ implies $F\equiv 0$. This follows at once from the previous proposition since $\|F\|_{PW_\alpha}=0$ implies $F(m,0)$ for every $m\in\mathbb Z$.
\end{proof}

The following result is the key to proving Theorem \ref{thm:PW-vero}.
\begin{thm}\label{thm:PW_decay_ascoli}
Let $F$ be a function in $PW_\alpha$. Then, for every $n\in\mathbb Z$,
\[
\lim_{|m|\to+\infty}F(m,n)=0.
\]
\end{thm}
The analogous result in the continuous setting follows by combining several classical results such as Phragm\'en--Lindel\"of-type theorems and Montel's theorem, see e.g. \cite[Theorem 12, Chapter 2]{Young}.
It is possible to exploit such technology in order to obtain even stronger results, such as the Plancherel--P\'olya inequality and its consequences, see e.g. \cite[Theorems 16 and 17, Chapter 2]{Young}.
In the discrete setting we partially have such a refined technology (see, e.g., \cite{ZD, Guadie, GM, BLMS}) at our disposal, but we do not actually need it.
Let us first recall the following result, which follows from a standard diagonalization process. 
\begin{thm}[see, e.g., {\cite[Theorem 7.23]{Rudin}}] \label{thm_rudin}
If $\{f_k\}_k$ is a pointwise bounded sequence of complex functions on a countable set $E$, then $\{f_k\}_k$ has a subsequence $\{f_{k_j}\}_{j}$ such that $\{f_{k_j}(x)\}_j$ converges for every  $x$ in $E$.
\end{thm}
Exploiting the above result we prove Theorem \ref{thm:PW_decay_ascoli}.
\begin{proof}[Proof of Theorem \ref{thm:PW_decay_ascoli}]
Let $n_0$ in $\mathbb N$ be fixed and define the strip $S_{n_0}=\{(m,n): m\in\mathbb Z, |n|\leqslant n_0\}$. Set $F_k(m,n)=F(m+k,n)$ and consider the sequence $\{F_k\}_{k\in\mathbb N}$. Such sequence is bounded on $S_{n_0}$ because of \eqref{eq:PW_p_i}. Hence, thanks to Theorem \ref{thm_rudin}, there exists a subsequence $\{F_{k_j}\}_{j\in\mathbb N}$ such that $\{F_{k_j}(m,n)\}_{j\in\mathbb N}$ converges for every $(m,n)$ in $S_{n_0}$. 
Set
\[
G(m,n)=\lim_{j\to+\infty}F_{k_j}(m,n)=\lim_{j\to+\infty} F(m+k_j,n).
\]
Then, $G$ is holomorphic in the strip $S_{n_0}$. In fact, for every $(m,n)\in S_{n_0-1}$ it holds 
\begin{align*}
G(m+1,&n+1)-G(m,n)+i\big(G(m,n+1)-G(m+1,n)\big)\\
&=\lim_{j\to+\infty}\Big( F_{k_j}(m+1,n+1)-F_{k_j}(m,n)+i\big(F_{k_j}(m,n+1)-F_{k_j}(m+1,n)\big)\Big)=0
\end{align*}
since each $F_{k_j}$ is entire. Notice now that for every $m\in\mathbb Z$,
\[
G(m,0)= \lim_{j\to+\infty} F_{k_j}(m,0)=\lim_{j\to+\infty}F(m+k_j,0)=0
\]
since $F$ satisfies \eqref{eq:PW_p_iii}.
Therefore, $G$ is a a holomorphic function in the strip $S_{n_0}$ which is identically zero at height $n=0$. Moreover, $G$ is clearly bounded on $S_{n_0+1}$. Adapting the proof of Proposition \ref{p:uniqueness} one can then deduce that such a function $G$ has to be identically zero on the whole strip $S_{n_0}$. 

We proved that the sequence $\{F_k\}_{k\in\mathbb N}$ admits pointwise converging subsequences in the strip $S_{n_0}$ and that each of such subsequences converges pointwise necessarily to  $0$ in $S_{n_0}$. It is then a standard fact to prove that the sequence $\{F_k\}_{k\in\mathbb N}$ itself actually pointwise converges to $0$ for every $(m,n)\in S_{n_0}$. Since the above arguments hold true for every strip $S_{n_0}$, we can conclude that 
\[
\lim_{k\to+\infty} F_k(m,n)=\lim_{k\to+\infty}F(m+k,n)=0
\]
for every $(m,n)\in\mathbb Z^2$. To conclude the proof it is enough to repeat the argument considering the sequence of functions 
$\{F_{-k}\}_{k\in\mathbb N}$, where $F_{-k}(m,n)=F(m-k,n)$.
\end{proof}
% ...

We are ready to prove Theorem \ref{thm:PW-vero}. 
\begin{proof}[Proof of Theorem \ref{thm:PW-vero}]
Let $f$ be in $L^2(\mathbb T)$ with $\supp f\subseteq D_\alpha$. Then,
\[
F(m,n)=\frac{1}{2\pi}\int_{\mathbb T} f(t) e_t(m,n)\, dt= \frac{1}{2\pi}\int_{D_\alpha} f(t) e_t(m,n)\, dt
\]
is well-defined and
\begin{align*}
|F(m,n)|
&\leqslant (2\pi)^{-1}  \|f\|_{L^2(\mathbb T)} |e_{-\alpha}(0,1)|^{|n|}.
\end{align*}
Thus, $F$ satisfies \eqref{eq:PW_p_i}. Moreover,
by Plancherel's theorem,
\[
\|F\|_{PW_\alpha}=\|f\|_{L^2(\mathbb T)}<+\infty.
\]
Hence, $F$ satisfies also \eqref{eq:PW_p_iii}. In conclusion, $F$ is a function in $PW_\alpha$ as we wished to show.

Conversely, let $F$ be in $PW_\alpha$. The function
\[
f(t)=\sum_{m\in\mathbb Z} F(m,0)e^{-imt}
\]
is a well-defined $L^2$ function. We want to prove that $\supp f_0\subseteq D_\alpha$; equivalently, that $f_0(\beta)=0$ whenever $\alpha<|\beta|<\pi-\alpha$. Let us first assume $\alpha<\beta<\pi-\alpha$ and let $C_R$ be the boundary of the discrete rectangle 
\[
\{(m,n)\in\mathbb Z^2: -R\leqslant R\leqslant m, \,0\leqslant n\leqslant R\}.
\]
Since $e_{\beta}$ and $F$ are both entire functions, by Proposition \ref{p:integral}, we have
\[
\int_{C_R} e_{\beta}: F=0,
\]
where the integral on the left-hand side is the contour integral defined in \eqref{eq:contour_integral}.
In particular
\begin{align}\label{eq:PW1}
\begin{split}
0
&=\sum_{m=-R}^{R-1}(e^{i\beta m}+e^{i\beta(m+1)})(F(m,0)+F(m+1,0))+4\int_{C'_{R}} e_{\beta}: F,
\end{split}
\end{align}
where $C_{R'}$ is the portion of the contour of $C_R$ in the upper half-plane. Taking the limit as $R\to+\infty$, thanks to the almost everywhere convergence of $L^2$ Fourier series, we get
\begin{align*}
0&= 2f(-\beta)(1+\cos(\beta))+4\lim_{R\to+\infty}\int_{C'_{R}} e_{\beta}: F.
\end{align*}
We claim that the limit of the contour integral along $C'_{R}$ is zero. Assuming the claim, we obtain
\[
0=2(1+\cos(\beta))f(-\beta).
\]
Hence, necessarily, it holds $f(-\beta)=0$ for almost every $\beta$ such that $\alpha<\beta<\pi-\alpha$. It remains to prove the claim. We have
\begin{align*}
4\int_{C_{R'}} e_{\beta}: F&=i\sum_{n=0}^{R-1}\big(e_{\beta}(R,n)+e_{\beta}(R,n+1)\big)\big(F(R,n)+F(R,n+1)\big)\\
&\quad -\sum_{m=-R}^{R-1}\big(e_{\beta}(m,R)+e_{\beta}(m+1,R)\big)\big(F(m,R)+F(m+1,R)\big)\\
&\quad -i\sum_{n=0}^{R-1}\big(e_{\beta}(-R,n+1)+e_{\beta}(-R, n)\big)\big(F(-R,n+1)+F(-R, n)\big)\\
&=I_R-I\!I_R-I\!I\!I_R.
\end{align*}
From \eqref{eq:PW_p_i} we get 
\begin{align*}
|I_R|& \leqslant c\sum_{n=0}^{R-1}\big(|e_{\beta}(0,1)|^n+|e_{\beta}(0,1)|^{n+1}\big)\big(|e_{-\alpha}(0,1)|^n+|e_{-\alpha}(0,1)|^{n+1}\big)\\
&\leqs c_{\alpha,\beta}\sum_{n=0}^{+\infty}\big(|e_{-\alpha}(0,1)||e_{\beta}(0,1)|\big)^n<+\infty
\end{align*}
since $|e_{-\alpha}(0,1)||e_{\beta}(0,1)|<1$. Therefore, by the dominated convergence theorem and Theorem \ref{thm:PW_decay_ascoli}, we get
\begin{align*}
\lim_{R\to+\infty} |I_R|&\leqs \lim_{R\to+\infty}\sum_{n=0}^{R-1}\big(|e_{\beta}(0,n)|+|e_{\beta}(0,n+1)|\big)\big(|F(R,n)|+|F(R,n+1)|\big)=0.
\end{align*}
With a similar argument as above it also follow that $I\!I\!I_R\to 0$ as $R\to+\infty$. Lastly, we have
\begin{align*}
|I\!I_R|&\leqs \sum_{m=-R}^{R-1}\big(|e_{\beta}(m,R)|+|e_{\beta}(m+1,R)|\big)\big(|F(m,R)|+|F(m+1,R)|\big)\\
&\leqs c \sum_{m=-R}^{R-1}\big(|e_{-\alpha}(0,1)||e_{\beta}(0,1)|\big)^R\to 0
\end{align*}
as $R\to+\infty$. In conclusion, we get the claim
\[
\lim_{R\to+\infty} \int_{C'_R} e_{\beta}(m,n): F(m,n) \, dz=0.
\]
If we repeat a similar argument considering the rectangle $R$ in the lower half-plane, then we conclude that $f(\beta)=0$ for almost every $\beta$ in $(\alpha,\pi-\alpha)$ as well. It remains to prove the representation formula \eqref{eq:PW}. Set
\[
G(m,n)=\frac{1}{2\pi}\int_{\mathbb T}f(t) e_t(m,n)\, dt.
\]
This function $G$ is well-defined on $\mathbb Z^2$ since $\supp f\subseteq D_\alpha$ and it is clearly entire. Moreover, 
\[
G(m,0)=\frac{1}{2\pi}\int_{\mathbb T} f(t)e^{imt}\, dt= F(m,0)
\]
for every $m\in\mathbb Z$. Therefore, by Proposition \ref{p:uniqueness}, the functions $G$ and $F$ coincide on the whole grid. 
\end{proof}

As an immediate consequence of Theorem \ref{thm:PW-vero} we obtain that $PW_\alpha$ is a reproducing kernel Hilbert space.  
\begin{cor}\label{kernel}
$PW_\alpha$ is a reproducing kernel Hilbert space. Its reproducing kernel is given by  
\begin{align*}
K_{(m,n)}(u,v)&=\frac{1}{2\pi}\int_{D_\alpha} e_t(u-m,v+n)\, dt=\frac{1}{2\pi}\int_{D_\alpha}
  e^{i(u-m)t}\bigg(\frac{1+ie^{it}}{i+e^{it}}\bigg)^{v+n}\, dt.
\end{align*}
\end{cor}
\begin{proof}
Theorem \ref{thm:PW-vero} guarantees that 
\[
F\mapsto \frac{1}{2\pi}\int_{\mathbb T}f(t) e_t(\cdot,\cdot)\, dt
\]
is a surjective isometry from $PW_\alpha$ onto the space
\[
\big\{f\in\ L^2(\mathbb T): \supp f\subseteq D_\alpha\big\}.
\]
Thus, $PW_\alpha$ is a Hilbert space endowed with the obvious inner product associated with the norm \eqref{PW_norm}. Moreover, by the reproducing formula \eqref{eq:PW}, we get
\begin{align*}
|F(m,n)|&=\frac{1}{2\pi}\bigg|\int_{D_\alpha} f(t) e_t(m,n)\, dt\bigg|\leqs c |e_{-\alpha}(0,1)|^{|n|}\|f\|_{L^2(\mathbb T)}= c|e_{-\alpha}(0,1)|^{|n|}\|F\|_{PW_\alpha}.
\end{align*}
Thus, the point-evaluation functionals are bounded on $PW_\alpha$. Equivalently, $PW_\alpha$ is a reproducing kernel Hilbert space. The kernel $K_{(m,n)}$ is immediately obtained with a standard argument. On one hand, we have
\begin{align}\label{eq:identity_kernel}
F(m,n)&=\langle F, K_{(m,n)}\rangle_{PW_\alpha}= \sum_{u\in\mathbb Z} F(u,0)\overline{K_{(m,n)}(u,0)}=\frac{1}{2\pi}\int_{D_\alpha} f(t) \overline{k_{(m,n)}(t)}\, dt,
\end{align}
where the last identity is guaranteed by Theorem \ref{thm:PW-vero} and $k_{(m,n)}$ is the Fourier transform of $K_{(m,n)}(\cdot,0)$ in $PW_\alpha$. On the other hand, matching \eqref{eq:identity_kernel} with \eqref{eq:PW} we obtain
\begin{align*}
\overline{k_{(m,n)}(t)}=  1_{D_\alpha}(t) e_t(m,n),
\end{align*}
where $1_{D_\alpha}$ is the characteristic function of $D_\alpha$.
Hence,
\begin{align*}
K_{(m,n)}(u,v)&=\frac{1}{2\pi}\int_{\mathbb T} k_{m,n}(t) e_t(u,v)\, dt
=\frac{1}{2\pi}\int_{D_\alpha} e_t(u-m,v+n)\, dt
\end{align*}
as we wished to show.
\end{proof}

Notice that the reproducing formula of $F$ via the kernel is of the form
\[
F(m,n)=\sum_{u\in\mathbb Z}F(u,0)K_{(m,n)}(u,0).
\]
In the special case of $v=-n$ the reproducing kernel is given by
\begin{align*}
K_{(m,n)}(u,-n )&=\frac{1}{2\pi}\int_{D_\alpha} e^{it(u-m)}=\frac{\alpha}{\pi}\big(1+(-1)^{u-m}\big)\sinc (\alpha(u-m))
\end{align*}
with $\sinc(x)=\frac{\sin(x)}{x}$. In particular, we obtain the identity
\begin{equation}\label{eq:kernel_sinc}
F(m,0)=\frac{\alpha}{\pi}\sum_{u\in\mathbb Z}F(u,0)\big(1+(-1)^{u-m}\big)\sinc(\alpha(u-m)), 
\end{equation}
 which is reminiscent of the classical cardinal sine series of the Whittaker--Kotelnikov--Shannon  sampling theorem.

Another immediate consequence of Theorem \ref{thm:PW-vero} is the following Plancherel--P\'olya type inequality.
\begin{cor}[Plancherel--P\'olya inequality]
Let $F$ be a function in $PW_\alpha$. Then,
\[
\sum_{m\in\mathbb Z} |F(m,n)|^2\leqs \bigg(\frac{\cos\alpha}{1-\sin\alpha}\bigg)^{2|n|}\|F\|_{PW_\alpha}^2 
\]
for every $n$ in $\mathbb Z$.
\end{cor}
\subsection{Proof of Theorem \ref{thm:decay_zero}}
The proof of \ref{thm:decay_zero} is essentially contained in the following lemma. 
\begin{lem}\label{lem:L1} 
 Let $F$ be a function in $PW_\alpha$. Assume that for every $k\in\mathbb N$ there exists $c_k>0$ such that
\[
|F(m,0)|\leqslant c_k (1+|m|)^{-k}.
\]
Then, for every $n$ in $\mathbb Z^+$ we have 
\begin{equation}\label{eq:value_y}
F(0,n)=-i F(0,n-1)-2i\sum_{k=1}^{+\infty}i^k F(k,n-1).
\end{equation}
Moreover, for $m$ in $\mathbb Z^+$ and every $n$ in $\mathbb Z^+$, 
\begin{equation}\label{eq:value_y_2}
|F(m,n)|\leqs 2^n \sum_{\ell=0}^{+\infty} (\ell+1)^{|n|-1}|F(m+\ell,0)| 
\end{equation}
and
\begin{equation}\label{eq:F_L1}
\sum_{m\in\mathbb Z^+}|F(m,n)|<+\infty.
\end{equation}
 \end{lem}

\begin{proof}[Proof of Lemma \ref{lem:L1}]
We proceed by induction on $n$ in $\mathbb  Z^+$. Assume that $n=1$. Then, by \eqref{formula_strati}, for every $m$  in $\mathbb Z^+$, we have
\begin{equation}\label{eq: F n=1}
    F(m,1)=(-i)^m\Big(F(0,1)+2i\sum_{k=1}^{m-1}i^k F(k,0)+iF(0,0)+i^{m+1}F(m,0)\Big).
\end{equation}
Set $m=4s+j$ with $j=0,1,2,3$. Since $F$ is rapidly decreasing at height $n=0$ we get
\begin{align}\label{eq:lim_uno}
\begin{split}
\lim_{s\to+\infty}F(4s+j,1)&=(-i)^{j}\Big(F(0,1)+2i\sum_{k=1}^{+\infty} i^kF(k,0)+iF(0,0)\Big)=:(-i)^j A_0.
\end{split}
\end{align}
Suppose that $A_0\neq 0$. Then,
at level $n=2$, by \eqref{formula_strati}, we have
\begin{align}\label{estimate}
\begin{split}
F(4m+1,2)&=(-i)\Big(F(0,2)+2i\sum_{k=1}^{4m}i^k F(k,1)+iF(0,1)+i^{4m+2}F(4m+1,1)\Big)\\
&=-iF(0,2)+2i\sum_{k=0}^{m-1}\beta_k+F(0,1)+iF(4m+1,1),
\end{split}
\end{align}
where
\[
\beta_k=iF(4k+1,1)+i^{2}F(4k+2,1)+i^{3}F(4k+3,1)+F(4k+4,1). 
\]
Notice that  from \eqref{eq:lim_uno} it follows
\begin{equation}\label{eq:beta_k_lim}
\lim_{k\to+\infty} \beta_k= 4A_0,
\end{equation}
while from \eqref{estimate} we get 
\begin{equation}\label{bound}
\Big|\sum_{k=0}^{m-1}\beta_k\Big|\leqs M_1+M_2<+\infty,
\end{equation}
where $M_n=\sup_{m\in\mathbb Z}|F(m,n)|$. Combining \eqref{eq:beta_k_lim} and \eqref{bound} one deduces that $\beta_k\to 0$ as $k\to+\infty$, that is, $A_0=0.$

Thus, from 
\eqref{eq:lim_uno} we obtain
\[
F(0,1)=-iF(0,0)-2i\sum_{k=1}^{+\infty} i^k F(k,0),
\]
which is \eqref{eq:value_y} for $n=1$. Notice that, plugging \eqref{eq:value_y} in \eqref{eq: F n=1}, we obtain
\begin{align*}
F(m,1)&=(-i)^m\Big(-iF(0,0)-2i\sum_{k=1}^{+\infty} i^k F(k,0)+2i\sum_{k=1}^{m-1}i^k F(k,0)+iF(0,0)+i^{m+1}F(m,0)\Big)\\
&=(-i)^m \Big(-2i\sum_{k=m}^{+\infty} i^k F(k,0)+i^{m+1}F(m,0)\Big). 
\end{align*}
Hence,
\[
|F(m,1)|\leqs 2 \sum_{k=m}^{+\infty} |F(k,0)|,
\]
which is \eqref{eq:value_y_2}.
Moreover, 
\begin{align}\label{eq:F1_L1}
\begin{split}
\sum_{m=1}^{+\infty}|F(m,1)|&\leqs 2 \sum_{m=1}^{+\infty}\sum_{k=m}^{+\infty}|F(k,0)|= 2 \sum_{k=1}^{+\infty}\sum_{m=1}^{k}|F(k,0)|\leqs 2 \sum_{k=1}^{+\infty}k|F(k,0)|<+\infty,
\end{split}
\end{align}
where the la series converges since $F(\cdot,0)$ is rapidly decreasing. In conclusion, \eqref{eq:F_L1} holds as well for $n=1$.
Let us now assume that \eqref{eq:value_y}, \eqref{eq:value_y_2} and \eqref{eq:F_L1}   hold true at height $n-1$ and prove them at height $n$. 
Equation \eqref{eq:value_y} follows by adapting the same argument as above since it only requires the absolute convergence of the series $\sum_{k=1}^{+\infty} i^kF(k,n-1)$ and this is guaranteed by \eqref{eq:F_L1}
 and the inductive hypothesis. By plugging \eqref{eq:value_y} in \eqref{formula_strati} and the inductive hypothesis, we get
 \begin{align*}
|F(m,n)|&=\Big|-2i\sum_{k=m}^{+\infty}i^k F(k,n-1)+i^{m+1} F(m,n-1)\Big|\\
&\leqslant 2\sum_{k=m}^{+\infty}|F(k,n-1)|\\
&\leqs 2^n\sum_{k=m}^{+\infty} \sum_{\ell=0}^{+\infty}(\ell+1)^{n-2}|F(\ell+k,0)|\\
&\leqs c_n\sum_{\ell=m}^{+\infty} (\ell-m+1)^{n-1} |F(\ell,0)|\\
&= 2^n \sum_{\ell=0}^{+\infty}(\ell+1)^{n-1}|F(\ell+m,0)|.
 \end{align*}
Thus, \eqref{eq:value_y_2} holds. Finally,
\begin{align*}
\sum_{m=1}^{\infty}|F(m,n)|&\leqs c_n \sum_{m=1}^{\infty}\sum_{\ell=0}^{+\infty}(\ell+1)^{n-1}|F(\ell+m,0)|\\
&= 2^n \sum_{\ell=1}^{\infty}\sum_{m=1}^{\ell}(\ell-m+1)^{n-1}|F(\ell,0)|\\
&\leqs c_n \sum_{\ell=1}^{\infty}\ell^n |F(\ell,0)|<+\infty
\end{align*}
since $F$ is rapidly decreasing at level $n=0$. Thus, \eqref{eq:F_L1} holds as well and the proof is complete.
\end{proof}

We now prove Theorem \ref{thm:decay_zero}.

\begin{proof}[Proof of Theorem \ref{thm:decay_zero}]
Applying the previous lemma to
$G(m,n)=(-1)^{m+n}F(-m,n)$, for every $n$  in $\mathbb Z^+$, we obtain
\[
\sum_{m\in\mathbb Z^+}|F(-m,n)|<+\infty. 
\]
This and \eqref{eq:F_L1} imply
\[
\lim_{|m|\to+\infty}|F(m,n)|=0
\]
for every positive $n$ in $\mathbb Z^+$. It remains to prove the result for $n$  in $\mathbb Z^-$, but it follows similarly as above applying Lemma \ref{lem:L1} to $H(m,n)=(-1)^{m+n}F(m,-n)$ and $L(m,n)=F(-m,-n)$, respectively. We omit the details.

\end{proof}

\section{Sampling in $PW_\alpha$}\label{s:sampling}
In this section we address the problem of a sufficient condition for a sequence to be sampling for $PW_\alpha$. However, before proving Theorem \ref{thm: Sampling}, we also easily derive a necessary condition 
by applying a very general result proved in \cite{RomeroLondon}.

\subsection{Necessary condition for sampling}
We consider $\mathbb{Z}$ as a metric measure space, endowed with the counting measure and the graph distance, whose balls are denoted by $B(m,r)$. One can easily verify that such metric measure space satisfies the hypothesis of \cite[Theorem 1.1]{RomeroLondon}. Hence, if $\Lambda\subseteq \mathbb Z$ is a sampling sequence for $PW_\alpha$, then the aforementioned result assures that the lower Beurling density satisfies the inequality
\begin{equation}\label{eq:sampling_nec}
D^-(\Lambda):= \liminf_{r\to+\infty}\inf_{m\in\mathbb{Z}}\frac{\#(\Lambda\cap B(m,r))}{\#B(m,r)}\geq \liminf_{r\to+\infty}\inf_{m\in\mathbb{Z}}\frac{1}{\#B(m,r)}\sum_{n\in B(m,r)}K_{(n,0)}(n,0)=\frac {2\alpha}{\pi},
\end{equation}
where the last identity follows from the fact that $K_{(n,0)}(n,0)=2\alpha/\pi$ for every $n\in\mathbb{Z}$, see \eqref{eq:kernel_sinc}. 

To better compare~\eqref{eq:sampling_nec} with the condition given in  Theorem~\ref{thm: Sampling} let us also assume that $\Lambda$ has a very easy configuration. Namely, let us assume that $\Lambda$ is of the form
\[
\Lambda=\delta_e\bZ\sqcup(n+\delta_o\bZ).
\]
where $\delta_e$ and $\delta_o$ are even integers and $n$ is some odd integer. 
We now want to compute $D^-(\Lambda)$. Notice that given $m\in\bZ$ and $r>0$,
\[
 2\Bigl\lfloor\frac r{\delta_e}\Bigr\rfloor\leqs\#\bigl(\Lambda\cap B(m,r)\cap 2\mathbb Z\bigr)\leqslant2\left(\Bigl\lfloor\frac r{\delta_e}\Bigr\rfloor+1\right)
 \]
 and similarly for $\#\bigl(\Lambda\cap B(m,r)\cap (2\mathbb Z+1)\bigr)$. Hence,

 \[
 2\left(\Bigl\lfloor\frac r{\delta_e}\Bigr\rfloor+\Bigl\lfloor\frac r{\delta_o}\Bigr\rfloor\right)\leqs\#\bigl(\Lambda\cap B(m,r)\bigr)\leqslant2\left(\Bigl\lfloor\frac r{\delta_e}\Bigr\rfloor+\Bigl\lfloor\frac r{\delta_o}\Bigr\rfloor+2\right).
 \]
However, we know that $\# B(m,r)=2r+1$. So that
\begin{align*}
    \frac 1{\delta_e}+\frac 1{\delta_o}=\lim_{r\to\infty}\frac 2{2r+1}\Bigl(\Bigl\lfloor\frac r{\delta_e}\Bigr\rfloor+\Bigl\lfloor\frac r{\delta_o}\Bigr\rfloor\Bigr)\leqslant D^-(\Lambda)\leqslant \lim_{r\to\infty}\frac 2{2r+1}\Bigl(\Bigl\lfloor\frac r{\delta_e}\Bigr\rfloor+\Bigl\lfloor\frac r{\delta_o}\Bigr\rfloor+2\Bigr)=\frac 1{\delta_e}+\frac 1{\delta_o},
\end{align*}
from which we conclude that if $\Lambda$ is a sampling set for $PW_\alpha$ then
\begin{equation}\label{eq:necessaria_delte}
\frac 1{\delta_e}+\frac 1{\delta_o}\geq \frac {2\alpha}\pi. 
\end{equation}

\subsection{Sufficient condition for sampling} 
In this last section we prove Theorem \ref{thm: Sampling}.  We apply well-established classical results to  provide a reconstructive algorithm, from which Theorem \ref{thm: Sampling} will follow. The proof we provide is a discrete adaptation of the approach presented in~\cite{G-irregular}, which has already been partially applied in a discrete periodic setting in~\cite{G-discrete}.
The main idea is to provide an approximation operator $A$ on $PW_\alpha$ such that the operator norm $\|\id-A\|_{PW_\alpha}$ is smaller than one. This would guarantee the invertibility of the operator $A$ and the validity of the following lemma.
\begin{lem}[{\cite[Lemma 3 ]{G-discrete}}]\label{l:G}
Let $A$ be a bounded operator on $PW_\alpha$ such that the operator norm $\|\id-A\|_{PW_\alpha}$ is smaller than one. Then $A$ is invertible and every $F\in PW_\alpha$ can be reconstructed by the iteration 
\[
\varphi_0= AF,\quad \varphi_{k+1}=\varphi_k-A\varphi_k,\quad F=\sum_{k=0}^{+\infty}\varphi_k
\]
with convergence in $PW_\alpha$.
\end{lem}
Let us now introduce an approximation operator $A$. Concerning notation, from now on, for every function $F$ defined on the grid $\mathbb Z^2$ we set $F(m)= F(m,0)$.  Consider $\Lambda\subseteq\mathbb Z$ and let $\{p_j\}_{j\in\mathbb Z}, \{q_j\}_{j\in\mathbb Z},\delta_e,
\delta_o$ be defined as in Theorem \ref{thm: Sampling}. 
Also, set 
\begin{equation}\label{eq:I_j}
I_j=2\mathbb Z\cap [p_j,p_{j+1})
\end{equation}
and 
\begin{equation}\label{eq:L_j}
L_j=(2\mathbb Z+1)\cap [q_j,q_{j+1}).
\end{equation}
Introduce the operators 
\begin{align*}
&T_j^eF(m)=\Biggl[F(p_j)+\frac{F(p_{j+1})-F(p_j)}{p_{j+1}-p_j}(m-p_j)\Biggr]1_{I_j}(m),\\
&
T_j^oF(m)=\Biggl[F(q_j)+\frac{F(q_{j+1})-F(q_j)}{q_{j+1}-q_j}(m-q_j)\Biggr]1_{L_j}(m)
\end{align*}
and set
\begin{align*}
&TF=\sum_{j\in\mathbb Z}T_j^eF+\sum_{j\in\mathbb Z}T_j^oF.
\end{align*}
Observe that, by definition, $TF(\lambda)=F(\lambda)$ for every $\lambda$ in $\Lambda$. 
Finally, set
 \[
 A=PT,
 \]
 where $P$ denotes the orthogonal projection $P\colon L^2(\mathbb Z)\to PW_\alpha$. Such a projection exists since $PW_\alpha$ is a closed subspace of $L^2(\mathbb Z)$, as a consequence of Theorem~\ref{thm:PW-vero}. 
We prove the following lemma.
\begin{lem}\label{l:norm_Id_A}
Let $\Lambda$ be a sequence of integers and set $\delta=\max(\delta_e,\delta_o)$. Then, for every $F\in PW_\alpha$,
\begin{equation}\label{eq:A_1}
\|AF\|_{PW_\alpha}\leqs 2\sqrt\delta \Big(\sum_{\lambda\in \Lambda} |F(\lambda)|^2 \Big)^{\frac12}
\end{equation}
and 
\begin{equation}\label{eq:A_2}
\|F-AF\|_{PW_\alpha}\leqslant \frac{\sin^2\alpha}{\sin^2(\frac{\pi}{\delta})}\|F\|_{PW_\alpha}.
\end{equation}
\end{lem}
From Lemma \ref{l:G}
 and Lemma \ref{l:norm_Id_A} the proof of Theorem \ref{thm: Sampling} follows immediately by a standard argument. 
 \begin{proof}[Proof of Theorem \ref{thm: Sampling}]
The right-hand inequality in \eqref{eq:sampling} is clearly trivial. The left-hand inequality follows combining Lemma \ref{l:G} and Lemma \ref{l:norm_Id_A}. In fact, notice that if $\delta=\max(\delta_e,\delta_0)<\pi/\alpha$, then $\frac{\sin^4\alpha}{\sin^4(\frac{\pi}{\delta})}$ is strictly smaller than one, so that the operator $A$ in invertible. Thus, 
\begin{align*}
\|F\|^2_{PW_\alpha}=\| A^{-1}AF\|^2_{PW_\alpha}\leqs 4\delta\|A^{-1}\|_{PW_\alpha}^2 \sum_{\lambda\in \Lambda}|F(\lambda)|^2.
\end{align*}
 \end{proof}

In only remains to prove  Lemma \ref{l:norm_Id_A}. To do so  we need two main ingredients, a ``step--two'' Bernstein inequality and a  discrete Wirtinger inequality. 

\begin{rmk}[Bernstein's inequality]\label{rmk:bernstein}
    Define the operator 
    \[\nabla_2F(m)=F(m)-F(m+2).
    \]
Let $f$ the Fourier transform on the torus of the function $F$ in  $PW_\alpha$. Then, the following Bernstein inequality holds true,
        \begin{align}\label{eq:Bernstein}
        \|\nabla_2F\|_{PW_\alpha} ^2&=\sum_{m\in\mathbb Z}|F(m+2)-F(m)|^2
        =\frac{1}{2\pi}\int_{D_\alpha}|f(t)|^2|e^{2it}-1|^2 dt
        \leqs 4\sin^2\alpha\|F\|^2_{PW_\alpha},
    \end{align}
  since $\sup_{t\in D_\alpha}|e^{2it}-1|^2=4\sin^2\alpha$. 
    In what follows, we are going to use the iterated operator 
    \[ \nabla_2^2F(m)=\nabla_2(\nabla_2F)(m)=F(m+4)-2F(m+2)+F(m) 
    \]
    and the inequality  \begin{equation}\label{eq: Bern}
    \|\nabla_2^2F\|_{PW_\alpha}\leqslant 4\sin^2\alpha\|F\|_{PW_\alpha}.\end{equation}

    Observe that by choosing the operator $\nabla_1F(m)=F(m+1)-F(m)$ one cannot obtain an inequality similar to \eqref{eq:Bernstein} with the constant depending on $\alpha$ since
    \begin{align*}
\|\nabla_1 F\|^2_{PW_\alpha}= \sum_{m\in\mathbb Z}|F(m+1)-F(m)|^2=\frac{1}{2\pi}\int_{D_\alpha} |f(t)|^2|e^{it}-1|^2\, dt
    \end{align*}
    and  $\sup_{t\in D_\alpha} |e^{it}-1|^2= 4$.
The fact that we have a meaningful Bernstein inequality for the operator $\nabla_2$ and not for $\nabla_1$ is the reason why the approximation operator $T$ consists of two parts, one concerning the even integers of the sequence $\Lambda$ and one concerning the odd ones. We are actually approximating separately the restriction of $F$ to the even integers and the odd integers. This is the reason why in our Theorem \ref{thm: Sampling} a condition that involves both $\delta_e$ and $\delta_o$ appears.
\end{rmk}

The second ingredient we need is a discrete version of the Wirtinger inequality.
\begin{thm}[{\cite[Theorem 11]{FTT}}]\label{thm: Wirt}
    Given $s(0),\dots,s(N)\in\mathbb C$ such that $s(0)=s(N)=0$, then     
    \[
    \sum_{\ell =0}^{N}|s(\ell) |^2\leqslant \frac 1{16\sin^4\bigl(\frac{\pi}{2N}\bigr)}\sum_{\ell =0}^{N-2}|\nabla_1^2s(\ell)|^2. \]
\end{thm}

We are ready to prove Lemma \ref{l:norm_Id_A}. 
\begin{proof}[Proof of Lemma \ref{l:norm_Id_A}]
Notice that 
\[
\|AF\|_{PW_\alpha}^2\leqs\|TF\|^2_{L^2}= \Big\|\sum_{j\in\mathbb Z} T^e_j F\Big\|^2_{L^2}+\Big\|\sum_{j\in\mathbb Z} T^o_j F\Big\|^2_{L^2}.
\]
Now,
\begin{align*}
    \Big\|\sum_{j\in\mathbb Z} T_j^eF\Big\|_2^2
   &= \sum_{m\in\mathbb Z} \Big|\sum_{j\in\mathbb Z} T^e_j F(m)\Big|^2
    \\
    &=\sum_{m\in\mathbb Z} \Big|\sum_{j\in\mathbb Z}\Big[F(p_j)+\frac{F(p_{j+1})-F(p_j)}{p_{j+1}-p_j}(m-p_j)\Big]1_{I_j}(m)\Big|^2\\
    &= \sum_{j\in\mathbb Z}\sum_{m\in I_j}\Big|F(p_j)+\frac{F(p_{j+1})-F(p_j)}{p_{j+1}-p_j}(m-p_j)\Big|^2\\
    &\leqs 2\sum_{j\in\mathbb Z}\sum_{m\in I_j}\left(\Big|\frac{F(p_{j})(p_{j+1}-m)}{p_{j+1}-p_j}\Big|^2+\Big|\frac{F(p_{j+1})(p_{j}-m)}{p_{j+1}-p_j}\Big|^2\right)\\
    &\leqs 2 \delta_e\sum_{j\in\mathbb Z}(|F(p_j)|^2+|F(p_{j+1})|^2).
\end{align*}
Similarly, we obtain
\[
 \Big\|\sum_{j\in\mathbb Z} T_j^oF\Big\|_2^2\leqs 2\delta_o\sum_{j\in\mathbb Z} (|F(q_j)|^2+|F(q_{j+1})|^2).
\]
In conclusion, 
\[
\|AF\|^2_{PW_\alpha}\leqs 4\delta \sum_{\lambda\in \Lambda} |F(\lambda)|^2
\]
as we wished to show. 
Let us now focus on $\|F-AF\|_{PW_\alpha
}$ with $F\in\ PW_\alpha$. Since $F\in PW_\alpha$, recalling \eqref{eq:I_j} and \eqref{eq:L_j}, we have
\[
F=PF=P\Big(\sum_{j\in\mathbb Z} F( 1_{I_j}+ 1_{L_j})\Big). 
\]
Then, 
\begin{align}\label{F_AF}
\begin{split}
\|F-AF\|^2_{PW_\alpha}&=\bigg\|P\bigg(\sum_{j\in\mathbb Z}(F-T^e_jF) 1_{I_j}+\sum_{j\in\mathbb Z}(F-T^o_jF) 1_{L_j}\bigg)\bigg\|^2_{PW_\alpha}\\
&\leqs \bigg\|\sum_{j\in\mathbb Z}(F-T^e_jF) 1_{I_j}+\sum_{j\in\mathbb Z}(F-T^o_jF) 1_{L_j}\bigg\|_{PW_\alpha}^2\\
&= \sum_{j\in\mathbb Z} \sum_{m\in I_j} \big|F(m)-T^e_jF(m)\big|^2+ \sum_{j\in\mathbb Z} \sum_{m\in L_j} \big|F(m)-T^o_jF(m)\big|^2.
\end{split}
\end{align}
Now,
\begin{align*}
    \sum_{m\in I_j} |F(m)-T^e_jF(m)|^2&= \sum_{m\in I_j} \left|F(m)-F(p_j)-\frac{F(p_{j+1})-F(p_j)}{p_{j+1}-p_j}(m-p_j)\right|^2\\
    &= \sum_{\ell=0}^{\frac 12(p_{j+1}-p_j)} \left|F(p_j+2\ell)-F(p_j)-\frac{F(p_{j+1})-F(p_j)}{p_{j+1}-p_j}(2\ell)\right|^2
\end{align*}
and use the discrete Wirtinger inequality as follows. Set 
\[
s_{j}^p(\ell)=F(p_j+2\ell)-F(p_j)-\frac{F(p_{j+1})-F(p_j)}{p_{j+1}-p_j}(2\ell),\qquad \ell\in\Big\{0,\dots,\frac 12(p_{j+1}-p_j)\Big\},
\]
and note that $s_j^p(0)=s_j^p(\frac 12(p_{j+1}-p_j))=0$. Hence, from Theorem \ref{thm: Wirt} we get
\begin{align*}
\sum_{m\in I_j} |F(m)-T^e_jF(m)|^2&=\sum_{\ell=0}^{\frac 12(p_{j+1}-p_j)} \left|s_j^p(\ell)\right|^2\leqslant \frac{1}{16\sin^4(\frac \pi{p_{j+1}-p_j})}\sum_{\ell=1}^{\frac 12(p_{j+1}-p_j)-2} \left|\nabla_1^2s_j^p(\ell)\right|^2,
\end{align*}
where, by direct computation,
\begin{align*}
    \nabla_1^2s_j^p(\ell) &=s_j^p(\ell+2)-2s_j^p(\ell+1)+s_j^p(\ell)\\
    &=F(p_j+2\ell+4)-2 F(p_j+2\ell+2)+F(p_j+2\ell)=\nabla_2^2F(p_j+2\ell).
\end{align*} 
Then,
\begin{align*}
\sum_{j\in\mathbb Z} \sum_{m\in I_j} \big|F(m)-T^e_jF(m)\big|^2&\leqs \sum_{j\in\mathbb Z}\frac{1}{16\sin^4(\frac \pi{p_{j+1}-p_j})}\sum_{\ell=1}^{\frac 12(p_{j+1}-p_j)-2} \left|\nabla_2^2F(p_j+2\ell)\right|^2\\
&\leqs \frac{1}{16\sin^4(\frac \pi{\delta_e})}\sum_{m\in2\mathbb Z} |\nabla^2_2 F(m)|^2. 
\end{align*}
Similarly, we obtain the analogous estimate for the odd integers. Namely,
\begin{align*}
\sum_{j\in\mathbb Z} \sum_{m\in L_j} \big|F(m)-T^o_jF(m)\big|^2
&\leqs \frac{1}{16\sin^4(\frac \pi{\delta_o})}\sum_{m\in2\mathbb Z+1} |\nabla^2_2 F(m)|^2. 
\end{align*}
Putting everything back together, from \eqref{F_AF} we get
\begin{align*}
\|F-AF\|^2_{PW_\alpha}&\leqs \frac{1}{16\sin^4(\frac{\pi}{\delta})} \sum_{m\in\mathbb Z}|\nabla^2_2 F(m)|^2\leqs \frac{\sin^4\alpha}{\sin^4 (\frac{\pi}{\delta})}\|F\|^2_{PW_\alpha},
\end{align*}
where we used the Bernstein inequality for the operator $\nabla_2$ (Remark \ref{rmk:bernstein}). 
\end{proof}

\begin{rmk}
Our sufficient sampling condition is not sharp. Notice that  $\max(\delta_e,\delta_o)<\pi/\alpha$ can be reformulated as $\min(1/\delta_e,1/\delta_0)>\alpha/\pi$, which clearly is a condition stronger than \eqref{eq:necessaria_delte}.
\end{rmk}

\section{Final Remarks}\label{s:remarks}
   A notion of Paley-Wiener spaces in discrete settings, specifically on combinatorial graphs, has been studied by Pesenson in the celebrated work~\cite{Pesenson} via spectral theory. We now compare $PW_\alpha$ with the spaces defined by Pesenson. 
Since $PW_\alpha$ can be identified with a closed subspace of $L^2(\mathbb Z)$ via Theorem \ref{thm:PW-vero}, we should compare $PW_\alpha$ with Pesenson's space on the graph $\mathbb Z$, which, following Pesenson's notation, we denote by $\mathcal P\mathcal W_\omega(\mathbb Z)$. In his work Pesenson never mention discrete holomorphicity; he only relies on spectral theory and provide the following characterization of $\mathcal P\mathcal W_\omega(\mathbb Z)$. As usual, in what follows we are identifying the torus $\mathbb T=\mathbb R/2\pi\mathbb Z$ with the interval $[-\pi, \pi)$.
\begin{thm}{{\cite[Theorem 5.1]{Pesenson}}}\label{thm:pesenson}
A function $F$ in $L^2(\mathbb Z)$ belongs to $\mathcal P\mathcal W_\omega(\mathbb Z)$, $0<\omega<2$, if and only if the Fourier transform $f$ of $F$ is a function on the torus $\mathbb T$ supported on the set
\[
\{t: |t|\leqs 2\arcsin\sqrt{\omega/ 2}\}.
\]
\end{thm}
Notice that $2\arcsin\sqrt{\omega/2}$ is in $(0,\pi)$. Thus, it is clear from Theorem \ref{thm:PW-vero} and Theorem \ref{thm:pesenson} that $PW_\alpha$ and $\mathcal P\mathcal W_\omega$ are not comparable in general. Nevertheless, we note the following. 
Given $F$ in $PW_\alpha$ set $F^e(m)=F(2m)$, where, again, we write $F(2m)$  meaning $F(2m,0)$. Then,
\begin{align*}
F^e(m)=\frac{1}{2\pi}\int_{D_\alpha} f(t)e^{i2mt}\, dt=\frac{1}{2\pi}\int_{2D_\alpha} f^e(t)  e^{imt}\, dt,
\end{align*}
   where $f^e(t)= f(t/2)/2$. Such function is supported on the set
   \[
   2D_\alpha= \{t:|t|<2\alpha\}.
   \]
   Thus, the function $F^e$ belongs to $\mathcal P\mathcal W_{\omega_\alpha}$ with $\omega_\alpha=2\sin^2(\alpha)$. The same holds true for the function $F^o(m)=F(2m+1)$. Therefore, to obtain a sampling result for $PW_\alpha$ we could apply Pesenson's sampling results for the space $\mathcal P\mathcal W_{\omega_\alpha}$ to the functions $F^e$ and $F^o$.  
According to \cite[Theorem 1.4]{Pesenson}, a function $F$ in $\mathcal P\mathcal W_{\omega_\alpha}$ is uniquely determined by its values on a set $U=\mathbb Z\backslash S$, 
   where $S$ is a finite or infinite union of sets $S_j$ of successive vertices such that
    \[\# S_j<\frac \pi{2\arcsin\sqrt{\frac{\omega_\alpha} 2}}-1=\frac{\pi}{2\alpha}-1,
    \]
     which, in other words, means that the maximum distance between two consecutive sampling points is $\pi/(2\alpha)$. 
       Hence, to sample $F$ in $PW_\alpha$ 
     by simultaneously sampling $F^e$ and $F^o$ we would need $\max(\delta_e,\delta_o)<\pi/\alpha$, which is exactly our hypothesis in Theorem \ref{thm: Sampling}.

\bibliographystyle{amsalpha}
\bibliography{DHFbib}
\end{document}